\documentclass[english,10pt]{amsart}
\usepackage[english]{babel}

\usepackage[matrix,arrow]{xy}
\xyoption{all}
\usepackage{amscd,amssymb,amsfonts,amsmath}

\usepackage{graphics}
\usepackage{epsfig}
\usepackage{mathrsfs}

\newcommand\mypagesizel{
\textwidth= 6.5in
\textheight=9in
\voffset-.55in
\hoffset -0.75in
\marginparwidth=56pt
}

\newcommand{\Pic}{\textup{Pic}}

\newcommand{\Supp}{\textup{Supp}}

\newcommand{\Chow}{\textup{Chow}}

\newcommand{\p}[0]{{\mathbb P}}
\newcommand{\rat}[0]{\operatorname{RatCurves}^n}

\renewcommand{\phi}{\varphi}
\newcommand{\into}{\hookrightarrow}
\newcommand{\map}{\dashrightarrow}

\renewcommand{\le}{\leqslant}
\renewcommand{\ge}{\geqslant}

\newcommand{\bQ}{\mathbb{Q}}

\newcommand{\bZ}{\mathbb{Z}}

\newcommand{\sA}{\mathscr{A}}

\newcommand{\sC}{\mathscr{C}}

\newcommand{\sE}{\mathscr{E}}
\newcommand{\sF}{\mathscr{F}}
\newcommand{\sG}{\mathscr{G}}
\newcommand{\sH}{\mathscr{H}}
\newcommand{\sI}{\mathscr{I}}

\newcommand{\sK}{\mathscr{K}}
\newcommand{\sL}{\mathscr{L}}
\newcommand{\sM}{\mathscr{M}}
\newcommand{\sN}{\mathscr{N}}
\newcommand{\sO}{\mathscr{O}}

\newcommand{\sQ}{\mathscr{Q}}

\mypagesizel

\newtheorem{thm}{Theorem}[section]
\newtheorem*{thm*}{Theorem}

\newtheorem{lemma}[thm]{Lemma}
\newtheorem{cor}[thm]{Corollary}

\newtheorem{prop}[thm]{Proposition}

\theoremstyle{definition}
\newtheorem{defn}[thm]{Definition}

\newtheorem{say}[thm]{}

\newtheorem{notation}[thm]{Notation}

\newtheorem{defn-thm}[thm]{Definition-Theorem}

\newtheorem{rem}[thm]{Remark}

\theoremstyle{remark}

\newtheorem*{not-and-def}{Notation and definitions}

\numberwithin{equation}{section}

\begin{document}

\title[On Fano foliations 2]{On Fano foliations 2}

\author{Carolina \textsc{Araujo}} 

\address{\noindent Carolina Araujo: IMPA, Estrada Dona Castorina 110, Rio de
  Janeiro, 22460-320, Brazil} 

\email{caraujo@impa.br}

\author{St\'ephane \textsc{Druel}}

\address{St\'ephane Druel: Institut Fourier, UMR 5582 du
  CNRS, Universit\'e Grenoble 1, BP 74, 38402 Saint Martin
  d'H\`eres, France} 

\email{druel@ujf-grenoble.fr}

\thanks{The first named author was partially supported by CNPq and Faperj Research 
  Fellowships}

\thanks{The second named author was partially supported by the CLASS project of the 
A.N.R}


\begin{abstract}
In this paper we pursue the study of mildly singular del Pezzo foliations on complex projective manifold
started in \cite{fano_fols}. 
\end{abstract}

\maketitle

%\abstract{•}

\tableofcontents

%
%      SECTION 1 
%

\section{Introduction}

In recent years, techniques from higher dimensional algebraic geometry, specially from 
the minimal model program, have been successfully applied to the study of global properties of 
holomorphic foliations. 
This led, for instance, to the birational classification of foliations by curves on surfaces in \cite{brunella}.  
Motivated by these developments,  we initiated in \cite{fano_fols} a systematic study of \emph{Fano foliations}.
These are holomorphic foliations $\sF$ on complex projective manifolds with ample anti-canonical class  
$-K_{\sF}$. 
One special property of Fano foliations is that their leaves are always covered by rational curves,
even when these leaves are not algebraic (see for instance \cite[Proposition 7.5]{fano_fols}).

The index $\iota_{\sF}$ of a Fano foliation $\sF$ on a  complex projective manifold $X$
is the largest integer dividing $-K_{\sF}$ in $\Pic(X)$.
In analogy with Kobayachi-Ochiai's theorem on the index of Fano manifolds (Theorem~\ref{kobayashi_ochiai}), 
we proved in \cite[Theorem 1.1]{adk08} that the index of a Fano foliation $\sF$ on a  complex projective manifold
is bounded above by its rank, $\iota_{\sF}\le r_{\sF}$. 
Equality holds if and only if $X\cong \p^n$ and 
$\sF$ is induced by a  linear projection $\p^n \dashrightarrow \p^{n-r_{\sF}}$.
Our expectation is that Fano foliations with large index are the simplest ones. 
So we proceeded to investigate the next case, namely 
Fano foliation $\sF$ of rank $r$ and index $\iota_{\sF}=r-1$.
We call such foliations \emph{del Pezzo foliations}, in analogy with the case of Fano manifolds.
In contrast to the case when $\iota_{\sF}= r_{\sF}$,
there are examples of del Pezzo foliations with non-algebraic leaves. 
For instance, let $\sC$ be a foliation by curves on $\p^k$  induced by a  global vector field. 
If we take this vector field to be general, then the leaves of $\sC$ are not algebraic.
Now consider a linear projection $\psi:\p^n\map \p^{k}$, with $n>k$, and let $\sF$
be the foliation on $\p^n$ obtained as pullback of $\sC$ via $\psi$.
It is a del Pezzo foliation on $\p^n$, and its leaves are not algebraic 
(see Theorem~\ref{lpt3fold}(2) for the complete classification of del Pezzo foliations on $\p^n$).
The first main result of \cite{fano_fols} says that these are the only examples. 

\begin{thm}[{\cite[Theorem 1.1]{fano_fols}}]\label{thma}
Let $\sF$ be a del Pezzo foliation on a complex projective manifold $X\not\cong \p^n$.
Then $\sF$ is algebraically integrable, and its general leaves are rationally connected.
\end{thm}

One of the main ingredients in our study of Fano foliations  is the notion 
of \emph{log leaf} for an algebraically integrable foliation.
Given an algebraically integrable foliation $\sF$ on a
complex projective manifold $X$,  denote by $\tilde{e}:\tilde F\to  X$ 
the normalization of the closure of a general leaf of $\sF$. 
There is a naturally defined effective Weil divisor $\tilde \Delta$ on $\tilde F$ such that
$K_{\tilde F}  +  \tilde \Delta = \tilde{e}^*K_{\sF}$
(see Definition~\ref{defn:log_leaf} for details). 
We call the pair $( \tilde F,  \tilde \Delta)$ a general \emph{log leaf} of $\sF$.
In \cite{fano_fols}, we used the log leaf to define new notions of singularities for 
algebraically integrable foliations, following
the theory of singularities of pairs from the minimal model program.
Namely, we say that $\sF$ has \emph{log canonical singularities along a general leaf} if 
$(\tilde F,\tilde \Delta)$ is log canonical.
By Theorem~\ref{thma}, these notions apply to 
del Pezzo foliations on projective manifolds $X\not\cong \p^n$.
In  \cite{fano_fols}, we established the following classification of del Pezzo foliations
with mild singularities.

\begin{thm}[{\cite[9.1 and Theorems 1.3, 9.2, 9.6]{fano_fols}}]\label{thmb}
Let $\sF$ be a del Pezzo foliation of rank $r$ on a complex projective manifold $X\not\cong \p^n$,
and suppose that $\sF$  has log canonical singularities and is locally free along a general leaf. 
Then 
\begin{itemize}
	\item either $\rho(X)=1$;
	\item or $r\le 3$ and $X$ is a $\p^m$-bundle over $\p^k$.
\end{itemize}
In the latter case,  one of the following holds.
\begin{enumerate}
\item $X\cong \p^1\times \p^k$, 
and $\sF$ is the pullback via the second projection of a  
foliation $\sO_{\p^k}(1)^{\oplus i}\subset T_{\p^k}$ for some $i\in\{1,2\}$ $(r\in\{2,3\})$.

\item There exist
\begin{itemize}
	\item an exact sequence of vector bundles  $0\to \sK\to \sE\to \sQ \to 0$ on $\p^k$; and 
	\item a foliation by curves $\sC\cong q^*\det(\sQ)\subset T_{\p_{\p^k}(\sK)}$, where 
		$q:\p_{\p^k}(\sK)\to \p^k$ denotes the natural projection;
\end{itemize}
such that $X\cong \p_{\p^k}(\sE)$, and $\sF$ is the pullback of $\sC$ via the relative linear projection $\p_{\p^k}(\sE)\map \p_{\p^k}(\sK)$. 
Moreover, one of the following holds.
			\begin{enumerate}
				\item $k=1$, $\sQ\cong \sO_{\p^1}(1)$, $\sK$ is an ample vector bundle
					such that $\sK\not\cong\sO_{\p^1}(a)^{\oplus m}$ for any integer $a$,
					and $\sE\cong\sQ \oplus\sK$ $(r=2)$.
				\item $k=1$, $\sQ\cong \sO_{\p^1}(2)$, $\sK\cong\sO_{\p^1}(a)^{\oplus m}$ 
					for some integer $a \ge 1$,
					and $\sE\cong\sQ \oplus\sK$ $(r=2)$.
				\item $k=1$, $\sQ\cong \sO_{\p^1}(1)\oplus \sO_{\p^1}(1)$, 
					$\sK\cong\sO_{\p^1}(a)^{\oplus (m-1)}$ 
					for some integer $a \ge 1$,
					and $\sE\cong\sQ \oplus\sK$ $(r=3)$.
				\item $k \ge 2$, $\sQ\cong \sO_{\p^k}(1)$, and $\sK$ is $V$-equivariant 	
					for some $V\in H^0\big(\p^k,T_{\p^k}\otimes\sO_{\p^k}(-1)\big)\setminus \{0\}$
					$(r=2)$.
			\end{enumerate}
Conversely, given $\sK$, $\sE$ and $\sQ$ satisfying any of the conditions above, there exists a del Pezzo foliation of that type.
\end{enumerate}
\end{thm}

The goal of the present paper is to continue the classification of del Pezzo foliations 
on Fano manifolds $X\not\cong \p^n$
having log canonical singularities and being locally free along a general leaf.
In view of Theorem~\ref{thmb}, we need to understand del Pezzo foliations on Fano manifolds with Picard number $1$.
Our main result is the following.

\begin{thm}\label{thm:main} Let $\sF$ be a del Pezzo foliation of rank $r\ge 3$ on an $n$-dimensional Fano manifold 
$X\not\cong \p^n$ with $\rho(X)=1$, and  suppose that $\sF$ has log canonical singularities and is locally free along  a general leaf. 
Then $X\cong Q^n$ and $\sF$ is induced by the restriction to $Q^n$ of a linear projection $\p^{n+1}\dashrightarrow \p^{n-r}$.
\end{thm}

\begin{rem}
Codimension $1$ del Pezzo foliations on Fano manifolds with Picard 
number $1$ were classified in \cite[Proposition 3.7]{loray_pereira_touzet_2}.
We extended this classification to mildly singular varieties, without restriction on the Picard number in 
\cite[Theorem 1.3]{codim_1_del_pezzo_fols}.

We also obtain a partial classification when $r=2$ (Proposition~\ref{classification_r=2}).
\end{rem}

In order to prove Theorem~\ref{thmb}, we consider a general log leaf $(\tilde F,\tilde \Delta)$ of $\sF$.
Under the assumptions of  Theorem~\ref{thmb},  $(\tilde F,\tilde \Delta)$ is a \emph{log del Pezzo pair}:
it is a log canonical pair of dimension $r$ satisfying $K_{\tilde F}  +  \tilde \Delta = (r-1)L$,
where $L$ is an ample divisor on $\tilde F$. 
The first step in the proof of Theorem~\ref{thmb} consists in classifying all log del Pezzo pairs. 
This is done in Section~\ref{ldP_pairs}, using Fujita's theory of  $\Delta$-genus.
Once we know the general log leaf $(\tilde F,\tilde \Delta)$ of $\sF$, 
we consider families of rational curves on $X$ that restrict to special families of rational curves on $\tilde F$.
The necessary results from the theory of rational curves are briefly reviewed in Section~\ref{Fanos}.
The idea is to use these families of rational curves to bound the index of $X$ from below. 
In order to obtain a good bound, we need to show that the dimension of these families of rational curves 
is big enough. Here enters a very special property 
of algebraically integrable Fano foliations having log canonical singularities along a general leaf: 
there is a common point contained in the closure of a general leaf (\cite[Proposition 5.3]{fano_fols}).
For our current purpose, we need the following strengthening of this result 
(see Definition~\ref{lccenter} for the notion of log canonical center).

\begin{prop}\label{weak_lemma:common_lc_center}
Let $\sF$ be an algebraically integrable Fano foliation on a complex projective manifold $X$
having log canonical singularities along a general leaf.
Then there is a closed  irreducible subset $T \subset X$ satisfying the following property.
For a general log leaf $(\tilde{F},\tilde{\Delta})$, there exists a log canonical center 
$S$ of $(\tilde{F},\tilde{\Delta})$ whose image in $X$ is $T$.
\end{prop}

When $r\ge 3$, this allows us to show that $\iota_X\ge n$, and then use Kobayashi-Ochiai's Theorem
(Theorem~\ref{kobayashi_ochiai}) to conclude that $X\cong Q^n$. 
The classification of del Pezzo foliations on $Q^n$ is established in Proposition~\ref{lemma:fols_in_Q^n}.

Proposition~\ref{weak_lemma:common_lc_center}  still holds in the more general setting of $\bQ$-\emph{Fano} 
foliations on possibly singular projective varieties. 
Since this may be useful in other situations, we present the theory of foliations on 
normal projective varieties in Section~\ref{sections:foliations}, and prove a  more general version of 
Proposition~\ref{weak_lemma:common_lc_center} (Proposition~\ref{lemma:common_lc_center}).

\

\noindent {\bf Notation and conventions.}
We always work over the field ${\mathbb C}$ of complex numbers. 
Varieties are always assumed to be irreducible.
We denote by $\textup{Sing}(X)$ the singular locus of a variety $X$.

Given a sheaf $\sF$ of $\sO_X$-modules on a variety $X$, we denote by $\sF^{*}$ the sheaf $\sH\hspace{-0.1cm}\textit{om}_{\sO_X}(\sF,\sO_X)$.
If $r$ is the generic rank of $\sF$, then we denote by $\det (\sF)$ the sheaf $(\wedge^r \sF)^{**}$.
If $\sG$ is another sheaf of $\sO_X$-modules on $X$, then we denote by  $\sF[\otimes]\sG$ the sheaf $(\sF\otimes\sG)^{**}$.

If $\sE$ is a locally free sheaf of $\sO_X$-modules on a variety $X$, 
we denote by $\p_X(\sE)$ the Grothendieck projectivization $\textup{Proj}_X(\textup{Sym}(\sE))$,
and by $\sO_{\p}(1)$ its tautological line bundle.

If $X$ is a normal variety, we denote by
$T_{X}$ the sheaf $(\Omega_{X}^1)^*$.

We denote by $Q^n$ a (possibly singular) quadric hypersurface in $\p^{n+1}$. Given an integer $d \ge 0$,
we denote by $\mathbb{F}_d$ the surface $\mathbb{P}_{\p^1}(\sO_{\p^1}\oplus\sO_{\p^1}(-d))$.
If moreover $d\ge 1$, we denote by 
$\p(1,1,d)$ the cone in $\p^{d+1}$ over the rational normal curve of degree $d$.

\

\noindent {\bf Acknowledgements.}
Much of this work was developed during the authors' visits to IMPA and Institut Fourier.
We would like to thank both institutions for their support and hospitality.

%
%      SECTION 2 
%

\section{Preliminaries}

\subsection{Fano manifolds and rational curves} \label{Fanos}

\begin{defn}
A \emph{Fano manifold} $X$ is a complex projective manifold whose anti-canonical class $-K_X$
is ample. 
The index $\iota_{X}$ of $X$ is the largest integer dividing $-K_{X}$ in $\Pic(X)$. 
\end{defn}

\begin{thm}[{\cite{kobayashi_ochiai}}]\label{kobayashi_ochiai}
Let $X$ be a Fano manifold of dimension $n\ge 2$ and index $\iota_{X}$.
Then $\iota_X\le n+1$, and equality holds if and only if $X\cong \p^n$.
Moreover,  $\iota_X= n$ if and only if $X\cong Q^n\subset \p^{n+1}$.
\end{thm}

Families of rational curves provide a useful tool in the study of Fano manifolds. 
Next we gather some results from the theory of rational curves. 
In what follows, rational curves are always assumed to be proper.
A \emph{family of rational curves} on a complex projective manifold $X$ is a 
closed irreducible subset of $\rat(X)$. 
We refer to \cite{kollar96} for details.

\begin{defn}\label{defn:minimal_curve}
Let $\ell \subset X$ be a rational curve on a complex projective manifold, 
and consider its normalization $f:\p^1\to X$.
We say that $\ell$ is \emph{free} if $f^*T_X$ is globally generated.
\end{defn}

\begin{say}\label{free&vfree}
Let $X$ be a complex projective manifold, and $\ell \subset X$ a free rational curve.
Let $x\in \ell$ be any point, and $H_x$ an irreducible component 
of the scheme $\rat(X,x)$ containing a point corresponding to $\ell$.
Then 
$$
\dim(H_x) \ = \ -K_X\cdot \ell \ - \ 2.
$$
\end{say}

\begin{notation}
Let $X$ be a Fano manifold with $\rho(X)=1$, and $\sA$ an ample 
line bundle on $X$ such that $\Pic(X)=\bZ[\sA]$.
For any proper curve $C\subset X$, we refer to $\sA\cdot C$ as the \emph{degree} of $C$. 
Rational curves of degree $1$ are called \emph{lines}. 
Note that if $C\subset X$ is a proper curve of degree $d$, then $\iota_X=\frac{-K_X\cdot C}{d}$.
\end{notation}

One can use free 
rational curves on Fano manifolds with Picard number $1$
to bound their index. 
The following is an immediate consequence of paragraph~\ref{free&vfree} above.

\begin{lemma}\label{lemma:bound_index}
Let $X$ be a Fano manifold with $\rho(X)=1$. 
Suppose that there is an $m$-dimensional family $V$ of rational curves of degree $d$ on $X$ such that:
		\begin{itemize}
			\item all curves from $V$ pass though some fixed point $x\in X$; and
			\item some curve from $V$ is free. 
		\end{itemize}
		Then $\iota_X\ge  \frac{m+2}{d}$.
\end{lemma}

\begin{rem}\label{rem:general_line_is_free}
Let $V$ be a family of rational curves on a  complex projective manifold $X$. 
To guarantee that some member of $V$ is a free curve, it is enough to show that some curve
from $V$ passes through a \emph{general} point of $X$. 
More precisely, let $H$ be an irreducible component of $\rat(X)$ containing $V$. 
It comes with universal family morphisms 

\centerline{
\xymatrix{
U \ar[r]^{e}\ar[d]_{\pi} & X \ , \\
H &
}
}
\noindent where $\pi:U\to H$ is a $\p^1$-bundle. 
Suppose that $e:U\to X$ is dominant.  
Then, by generic smoothness, there is a dense open subset $X^{\circ}\subset X$ 
over which $e:U\to X$ is smooth.
On the other hand, by \cite[Proposition II.3.4]{kollar96}, $e$ is smooth at a point $u\in \pi^{-1}(t)$ 
if and only if the rational curve $\ell_t=e\big(\pi^{-1}(t)\big)$ is free. 
\end{rem}

\subsection{Singularities of pairs} \label{sings_of_pairs}

We refer to \cite[section 2.3]{kollar_mori} and \cite[sections 2 and 4]{kollar_sing_mmp} for details.

\begin{defn}\label{Singularities of pairs}
Let $X$ be a normal projective variety, and
$\Delta=\sum a_i\Delta_i$ an effective $\bQ$-divisor on $X$, i.e., $\Delta$ is  a nonnegative $\bQ$-linear combination 
of distinct prime Weil divisors $\Delta_i$'s on $X$. 
Suppose that $K_X+\Delta$ is $\bQ$-Cartier, i.e.,  some nonzero multiple of it is a Cartier divisor. 

Let $f:\tilde X\to X$ be a log resolution of the pair $(X,\Delta)$. 
There are uniquely defined rational numbers $a(E_i,X,\Delta)$'s such that
$$
K_{\tilde X}\ +\ f^{-1}_*\Delta \ = \  f^*(K_X+\Delta)\ +\ \sum_{E_i}a(E_i,X,\Delta)E_i.
$$
The $a(E_i,X,\Delta)$'s do not depend on the log resolution $f$,
but only on the valuations associated to the $E_i$'s. 
The closed subvariety $f(E_i)\subset X$ is called the \emph{center of $E_i$ in $X$}.
It also depends only on the valuation associated to $E_i$.

For a prime divisor $D$ on $X$, we define $a(D,X,\Delta)$ to be the coefficient of $D$ in $-\Delta$.

We say that  the pair $(X,\Delta)$ is \emph{log canonical} if, for some log resolution 
$f:\tilde X\to X$ of $(X,\Delta)$, $a(E_i,X,\Delta)\ge -1$ 
for every $f$-exceptional prime divisor $E_i$.
If this condition holds for some log resolution of $(X,\Delta)$, then it holds for every  
log resolution of $(X,\Delta)$.
\end{defn}

\begin{defn} \label{lccenter}
Let $(X,\Delta)$ be a log canonical pair. We say that a closed irreducible subvariety $S\subset X$ is a 
\emph{log canonical center} of $(X,\Delta)$ if there is a divisor $E$ over $X$ with 
$a(E,X,\Delta)= -1$ whose center in $X$ is $S$. 
\end{defn}

\subsection{Polarized varieties and Fujita's $\Delta$-genus} \label{polarized_vars}

\begin{defn}\label{defn:polarized_vars}
A \emph{polarized variety} is a pair $(X,\sL)$ consisting of a normal projective variety $X$, 
and an ample line bundle $\sL$ on $X$.
\end{defn}

\begin{defn}[\cite{fujita75}]\label{defn:D-genus}
The \emph{$\Delta$-genus} of an $n$-dimensional polarized variety $(X,\sL)$ is defined by the formula:
$$
\Delta(X,\sL):=n+c_1(\sL)^n-h^0(X,\sL) \in \mathbb Z.
$$ 
\end{defn}

By \cite[Corollary 2.12]{fujita82b}, $\Delta(X,\sL) \ge 0$ for any polarized variety $(X,\sL)$.
Next we recall the classification of polarized varieties with $\Delta$-genus zero from \cite{fujita82b}.

\begin{thm}[\cite{fujita82b}]\label{thm:classification_delta_genus_zero}
Let $X$ be a normal projective variety of dimension $n\ge 1$, and $\sL$ an ample line bundle on $X$. Suppose that $\Delta(X,\sL)=0$.
Then one of the following holds.
\begin{enumerate}
\item $(X,\sL)\cong (\p^n,\sO_{\p^n}(1))$. 
\item $(X,\sL) \cong (Q^n,\sO_{Q^n}(1))$.
\item $(X,\sL)\cong (\p^1,\sO_{\p^1}(d))$, for some $d\ge 3$.
\item $(X,\sL)\cong(\p^2,\sO_{\p^2}(2))$.
\item $(X,\sL)\cong (\p_{\p^1}(\sE),\sO_{\p_{\p^1}(\sE)}(1))$,
where $\sE$ is an ample vector bundle on $\p^1$.
\item $\sL$ is very ample, and embeds $X$ as a cone over a projective polarized variety 
of type (3-5)  above. 
\end{enumerate}
\end{thm}

\subsection{Classification of log del Pezzo pairs}\label{ldP_pairs}

\begin{defn}
Let $X$ be a normal projective variety of dimension $n\ge 1$, and
$\Delta$ an effective $\bQ$-divisor on $X$.
We say that $(X,\Delta)$ is a \emph{log del Pezzo pair} if 
$(X,\Delta)$ is log canonical, and $-(K_X+\Delta)\equiv (n-1)c_1(\sL)$ for some 
ample line bundle $\sL$ on $X$.
\end{defn}

Using Fujita's classification of polarized varieties with $\Delta$-genus zero, we classify log del Pezzo pairs
in Theorem~\ref{thm:classification_del_pezzo} below.

\begin{lemma}\label{lemma:delta_genus}
Let $(X,\sL)$ be an $n$-dimensional polarized variety, with $n\ge 2$.
Let $(X,\Delta)$ be a log del Pezzo pair such that $\Delta\neq 0$ and 
$-(K_X+\Delta)\equiv (n-1)c_1(\sL)$. 
Then $\Delta(X,\sL)=0$, and 
$\Delta\cdot c_1(\sL)^{n-1}=2$.
\end{lemma}

\begin{proof}
We follow the line of argumentation in the proof of \cite[Lemma 1.10]{fujita80}.
Since 
$$
c_1(\sL^{\otimes t})\equiv K_X+\Delta+c_1(\sL^{\otimes n-1+t}),
$$
we have that $h^i(X,\sL^{\otimes t})=0$ for $i\ge 1$ and $t>1-n$ by \cite[Theorem 8.1]{fujinoMMPlc}.
Therefore $\chi(X,\sO_X)=1$ and $\chi(X,\sL^{\otimes t})=0$ for $2-n\le t \le -1$.
Hence, there are rational numbers $a$ and $b$ such that 
$$
\chi(X,\sL^{\otimes t})=\big(at^2+bt+n(n-1)\big)\frac{\prod_{j=1}^{n-2}(t+j)}{n!}.
$$
On the other hand, since $X$ is  normal, by Hirzebruch-Riemann-Roch,
$$
\chi(X,\sL^{\otimes t})=\frac{c_1(\sL)^n}{n!}t^n-\frac{1}{2(n-1)!}K_X\cdot c_1(\sL)^{n-1}t^{n-1}+o(t^{n-1}).
$$
Thus we have $a=c_1(\sL)^n$ and
$b=\frac{n}{2}\Delta\cdot c_1(\sL)^{n-1}
+(n-1)c_1(\sL)^{n}$. In particular, 
$$
h^0(X,\sL)=\chi(X,\sL)=n-1+c_1(\sL)^n
+\frac{1}{2}\Delta\cdot c_1(\sL)^{n-1}.
$$
One then computes that
$$
\Delta(X,\sL)=1-\frac{1}{2}\Delta\cdot c_1(\sL)^{n-1}.
$$
Since  $\Delta\neq 0$ and $\Delta(X,\sL) \ge 0$, we must have 
$\Delta(X,\sL) = 0$ and $\Delta\cdot c_1(\sL)^{n-1}=2$.
\end{proof}

\begin{thm}\label{thm:classification_del_pezzo}
Let $(X,\sL)$ be an $n$-dimensional polarized variety, with $n\ge 1$.
Let $(X,\Delta)$ be a log del Pezzo pair such that $\Delta$ is integral and nonzero, and 
$-(K_X+\Delta)\equiv (n-1)c_1(\sL)$. 
Then one of the following holds.
\begin{enumerate}
\item $(X,\sL,\sO_X(\Delta))\cong (\p^n,\sO_{\p^n}(1),\sO_{\p^n}(2))$. 
\item $(X,\sL,\sO_X(\Delta)) \cong (Q^n,\sO_{Q^n}(1),\sO_{Q^n}(1))$.
\item $(X,\sL,\sO_X(\Delta))\cong (\p^1,\sO_{\p^1}(d),\sO_{\p^1}(2))$, for some integer $d\ge 3$.
\item $(X,\sL,\sO_X(\Delta))\cong(\p^2,\sO_{\p^2}(2),\sO_{\p^2}(1))$.
\item $(X,\sL)\cong (\p_{\p^1}(\sE),\sO_{\p_{\p^1}(\sE)}(1))$  for an ample vector bundle $\sE$ on $\p^1$. 
Moreover, one of the following holds.
\begin{enumerate}
\item $\sE=\sO_{\p^1}(1)\oplus\sO_{\p^1}(a)$ for some $a\ge 2$, and $\Delta\sim_\mathbb{Z} \sigma+ f$
where $\sigma$ is the minimal section and $f$ is a fiber of $\p_{\p^1}(\sE) \to \p^1$.
\item $\sE=\sO_{\p^1}(2)
\oplus\sO_{\p^1}(a)$ for some $a\ge 2$, and $\Delta$ is a minimal section.
\item $\sE=\sO_{\p^1}(1)\oplus \sO_{\p^1}(1)
\oplus\sO_{\p^1}(a)$ for some $a \ge 1$, and $\Delta=\p_{\p^1}(\sO_{\p^1}(1)\oplus \sO_{\p^1}(1))$.
\end{enumerate}
\item $\sL$ is very ample, and embeds $(X,\Delta)$ as a cone over $\big((Z,\sM),(\Delta_Z, \sM_{|\Delta_Z})\big)$,
where $Z$ is smooth and $(Z,\sM,\Delta_Z)$ satisfies one of the conditions (3-5) above. 
\end{enumerate}
\end{thm}

\begin{proof}
By \cite[Theorem 0.1]{campana_koziarz_paun}, we must have 
$-(K_X+\Delta)\sim_\bQ (n-1)c_1(\sL)$.

If $n=1$, then $-K_X\sim_\bQ\Delta$ is ample, and hence
$(X,\sL,\sO_X(\Delta))$ satisfies one of conditions (1-3) in the statement of Theorem \ref{thm:classification_del_pezzo}.

Suppose from now on that $n\ge 2$.
By Lemma \ref{lemma:delta_genus}, $\Delta(X,\sL)=0$, and so
we can apply Theorem \ref{thm:classification_delta_genus_zero}. Notice that
if $(X,\sL)$ satisfies any of conditions (1-6) of Theorem \ref{thm:classification_delta_genus_zero}, then
$-(K_X+\Delta)\sim_\mathbb{Z} (n-1)c_1(\sL)$ since $X\setminus\textup{Sing}(X)$ is simply connected.

If $(X,\sL)$ satisfies any of conditions (1-4) of Theorem \ref{thm:classification_delta_genus_zero}, 
one checks easily that $(X,\sL,\Delta)$ satisfies one of conditions (1-4) in the statement of Theorem \ref{thm:classification_del_pezzo}.

Suppose that $(X,\sL)\cong (\p_{\p^1}(\sE),\sO_{\p_{\p^1}(\sE)}(1))$
for an ample vector bundle $\sE$ on $\p^1$, and write $\pi : X \to \p^1$ for the natural projection.
Then 
$$
\Delta\in \big|\sO_X(-K_X)\otimes\sL^{\otimes 1-n}\big|=\big|\sL\otimes \pi^*\big(\det(\sE^*)\otimes\sO_{\p^1}(2)\big)\big|.
$$
Write $\sE\cong\sO_{\p^1}(a_1)
\oplus\cdots\oplus\sO_{\p^1}(a_n)$, with $1\le a_1\le \cdots \le a_n$.
By the projection formula, $h^0\big(X,\sL\otimes \pi^*\big(\det(\sE^*)\otimes\sO_{\p^1}(2)\big)\big)
=h^0(\p^1,\sE\otimes\det(\sE^*)\otimes\sO_{\p^1}(2))$, hence we must have 
$a_1+\cdots+a_{n-1} \le 2$.
This implies that 
$(n,a_1,\ldots,a_{n-1})\in\{(2,1),(2,2),(3,1,1)\}$. 
Thus either $\sE$ satisfies condition  
(5a-c) in the statement of Theorem \ref{thm:classification_del_pezzo}, or 
$\sE=\sO_{\p^1}(1)\oplus\sO_{\p^1}(1)$, $\Delta\in|\sO_{\p_{\p^1}(\sE)}(1)|$, and hence $X$ 
satisfies condition  (2) with $n=2$.

Finally, suppose that $\sL$ is very ample, and embeds $X$ as a cone with vertex $V$ over a smooth  polarized variety $(Z,\sM)$
satisfying one of  conditions (3-5) in the statement of Theorem \ref{thm:classification_delta_genus_zero}.
Set $m:=\dim(Z)$ and $s:=n-m=\dim(V)+1$. 
Let $e : Y \to X$ be the blow-up of $X$ along $V$, with exceptional divisor $E$. 
We have
$Y\cong \p_Z(\sM\oplus \sO_Z^{\oplus s})$, with natural projection $\pi : Y\to Z$, 
and  tautological line bundle $\sO_Y(1) \cong e^*\sL$.
The exceptional divisor $E$ corresponds to the projection 
$\sM\oplus \sO_Z^{\oplus s}\twoheadrightarrow \sO_Z^{\oplus s}$.

Let $\Delta_Y$ be the strict transform of $\Delta$ in $Y$. 
We are done if we prove that $\Delta_Y=\pi^*\Delta_Z$ for some divisor $\Delta_Z$ on $Z$.

Write $\Delta_Y\sim_\mathbb{Z}\pi^*\Delta_Z+k E$ for some
integral divisor $\Delta_Z$ on $Z$, and some integer $k\ge 0$.
Let $\sigma : Z \to Y$ be the section of $\pi$
corresponding to a general surjection 
$\sM\oplus \sO_Z^{\oplus s}\twoheadrightarrow \sM$. 
Then $\sigma(Z)\cap E=\emptyset$, and $\sN_{\sigma(Z)/Y}\cong \sM^{\oplus s}$. 
Moreover, 
$(\sigma(Z),{\Delta_Y}_{|\sigma(Z)})$ is log canonical (see for instance \cite[Proposition 7.3.2]{kollar97}), 
and, by the adjunction formula, 
$-(K_{\sigma(Z)}+{\Delta_Y}_{|\sigma(Z)})\sim_\mathbb{Z} (m-1){c_1(\sO_Y(1))}_{|\sigma(Z)}$. 

We have $h^0(Y,\sO_Y(k E+\pi^*\Delta_Z)=h^0(Y,\sO_Y(\Delta_Y))\ge 1$. On the other hand, 
\begin{multline*}
$$h^0(Y,\sO_Y(k E+\pi^*\Delta_Z))=
h^0(Y,\sO_Y(k)\otimes\pi^*\sM^{\otimes -k}\otimes\sO_Y(\pi^*\Delta_Z))\\
=h^0(Z,S^{k}(\sM\oplus \sO_Z^{\oplus s})\otimes\sM^{\otimes - k}\otimes\sO_Z(\Delta_Z))=h^0(Z,S^{k}(\sO_Z\oplus {\sM_Z^{\otimes -1}}^{\oplus s})\otimes\sO_Z(\Delta_Z)).$$
\end{multline*}
We claim that $h^0(Z,\sM_Z^{\otimes -1}\otimes\sO_Z(\Delta_Z))= 0$.
Indeed, suppose that  $h^0(Z,\sM_Z^{\otimes -1}\otimes\sO_Z(\Delta_Z))\neq  0$. Then 
$-K_Z \sim_\mathbb{Z} \Delta_Z+(m-1)c_1(\sM)$ 
since
${\Delta_Y}_{|\sigma(Z)}\sim_\mathbb{Z} \big(\pi_{|\sigma(Z)}\big)^*\Delta_Z$, and hence
$-K_Z \ge m c_1(\sM)$.
Under these conditions,  \cite[Theorem 2.5]{codim_1_del_pezzo_fols} implies that 
$(Z,\sM,\sO_Z(\Delta_Z))$ is isomorphic to either  $(\p^m,\sO_{\p^m}(1),\sO_{\p^m}(2))$ or 
$(Q^m,\sO_{Q^m}(1),\sO_{Q^m}(1))$. 
This contradicts our current assumption that $(Z,\sM)$
satisfies one of  conditions (3-5) in the statement of Theorem \ref{thm:classification_delta_genus_zero},
and proves the claim.

Since $h^0(Z,\sM_Z^{\otimes -1}\otimes\sO_Z(\Delta_Z))= 0$, we must have
$h^0(Y,\sO_Y(k E+\pi^*\Delta_Z))=h^0(Z,\sO_Z(\Delta_Z))$. 
Thus, replacing $\Delta_Z$ with a suitable member of its linear system if necessary,
we may assume that
$\Delta_Y=\pi^*\Delta_Z+k E$, and hence $k=0$. 
Therefore $(X,\sL,\Delta)$ satisfies condition (6) in the statement of Theorem  
\ref{thm:classification_del_pezzo}.
\end{proof}

In dimension $2$, we have the following classification, without the assumption that $(X,\Delta)$ is log canonical.

\begin{thm}[{\cite[Theorem 4.8]{nakayama07}}]\label{nakayama}
Let $(X,\Delta)$ be a pair with $\dim(X)=2$ and $\Delta\neq 0$. 
Suppose that  $-(K_X+\Delta)$ is Cartier and ample.
Then one of the following holds.

\begin{enumerate}
\item $X\cong \p^2$ and $\deg(\Delta)\in\{1,2\}$.
\item $X\cong \mathbb{F}_d$ for some $d\ge 0$ and $\Delta$ is a minimal section.
\item $X\cong \mathbb{F}_d$ for some $d\ge 0$ and $\Delta\sim_\mathbb{Z} \sigma+ f$,
where $\sigma$ is a minimal section and $f$ a fiber of $\mathbb{F}_d \to \p^1$.
\item $X\cong \p(1,1,d)$ for some $d\ge 2$ and 
$\Delta\sim_\mathbb{Z} 2 \ell$ where $\ell$ is a ruling of the cone $\p(1,1,d)$.
\end{enumerate}
\end{thm}

%
%      SECTION 3 
%

\section{Foliations}\label{sections:foliations}

\subsection{Foliations and Pfaff fields}

\begin{defn}
Let $X$ be normal variety.
A \emph{foliation} on $X$ is a nonzero coherent subsheaf $\sF\subsetneq T_X$ such that
\begin{itemize}
	\item $\sF$ is closed under the Lie bracket, and
	\item $\sF$ is saturated in $T_X$ (i.e., $T_X / \sF$ is torsion free).
\end{itemize}

The \textit{rank} $r$ of $\sF$ is the generic rank of $\sF$. The \emph{codimension} of $\sF$ is $q=\dim(X)-r$.

The \textit{canonical class} $K_{\sF}$ of $\sF$ is any Weil divisor on $X$ such that 
$\sO_X(-K_{\sF})\cong \det(\sF)$. 
\end{defn}

\begin{defn}\label{def:gorenstein}
A foliation $\sF$ on a normal variety is said to be \emph{$\bQ$-Gorenstein} if 
its canonical class $K_{\sF}$ is $\bQ$-Cartier.
\end{defn}

\begin{defn}\label{def:pfaff}
Let $X$ be a variety, and $r$ a  positive integer.
A  \emph{Pfaff field of rank r} on $X$ is a nonzero map
$\eta : \Omega^r_X\to \sL$, where  $\sL$ is a reflexive sheaf of rank 1
on $X$ such that $\sL^{[m]}$ is invertible for some integer $m\ge 1$. 

The \textit{singular locus} $S$ of $\eta$
is the closed subscheme of $X$ whose ideal sheaf $\sI_S$ is the image of
the induced map $\Omega^r_X[\otimes] \sL^*\to \sO_X$.  
\end{defn}

Notice that a $\bQ$-Gorenstein foliation $\sF$ of rank $r$  on normal variety $X$ naturaly gives rise to a Pfaff field of rank $r$ on $X$:
$$
\eta:\Omega_X^r=\wedge^r(\Omega_X^1) \to \wedge^r(T_X^*) \to \wedge^r(\sF^*) \to \det(\sF^*)\cong\det(\sF)^{*}\cong\sO_X(K_\sF).
$$

\begin{defn}\label{defn:sing_locus} \label{def:regular}
Let  $\sF$ be a $\bQ$-Gorenstein foliation on a normal variety $X$.
The \textit{singular locus} of  $\sF$ is defined to be
the singular locus $S$ of the associated Pfaff field. 
We say that $\sF$ is \textit{regular at a point} $x\in X$ if $x\not\in S$. 
We say that $\sF$ is \textit{regular} if $S=\emptyset$.
\end{defn}

Our definition of Pfaff field is more general than the one usually found in the literature, where $\sL$ is required to be
invertibile. This generalization is needed in order to treat  $\bQ$-Gorenstein foliations whose canonical classes are not Cartier.

\begin{say}[Foliations defined by $q$-forms] \label{q-forms}
Let $\sF$ be a codimension $q$ foliation on an $n$-dimenional normal variety $X$.
The \emph{normal sheaf} of $\sF$ is $N_\sF:=(T_X/\sF)^{**}$.
The $q$-th wedge product of the inclusion
$N^*_\sF\into (\Omega^1_X)^{**}$ gives rise to a nonzero global section 
 $\omega\in H^0\big(X,\Omega^{q}_X[\otimes] \det(N_\sF)\big)$
 whose zero locus has codimension at least $2$ in $X$. 
Moreover, $\omega$ is \emph{locally decomposable} and \emph{integrable}.
To say that $\omega$ is locally decomposable means that, 
in a neighborhood of a general point of $X$, $\omega$ decomposes as the wedge product of $q$ local $1$-forms 
$\omega=\omega_1\wedge\cdots\wedge\omega_q$.
To say that it is integrable means that for this local decomposition one has 
$d\omega_i\wedge \omega=0$ for every  $i\in\{1,\ldots,q\}$. 
The integrability condition for $\omega$ is equivalent to the condition that $\sF$ 
is closed under the Lie bracket.

Conversely, let $\sL$ be a reflexive sheaf of rank $1$ on $X$, and 
$\omega\in H^0(X,\Omega^{q}_X[\otimes] \sL)$ a global section
whose zero locus has codimension at least $2$ in $X$.
Suppose that $\omega$  is locally decomposable and integrable.
Then  the kernel
of the morphism $T_X \to \Omega^{q-1}_X[\otimes] \sL$ given by the contraction with $\omega$
defines 
a foliation of codimension $q$ on $X$. 
These constructions are inverse of each other. 
 \end{say}

% SUBSECTION -  Algebraically integrable foliations

\subsection{Algebraically integrable foliations}\label{section:algebraic_foliations} \

\medskip

Let $X$ be a normal projective variety, and $\sF$ a foliation on $X$.
In this subsection we assume that $\sF$ is \emph{algebraically integrable}. 
%This means that  the leaf of $\sF$ through a general point of $X$ is an algebraic variety. 
This means that  $\sF$ is the relative tangent sheaf to a dominant rational map $\varphi:X\map Y$
with connected fibers. 
In this case, by a general leaf of $\sF$ we mean the fiber of $\varphi$ over a general point of $Y$.
We start by defining the notion of  \emph{log leaf} 
when $\sF$ is moreover  $\bQ$-Gorenstein.
It plays a key role in our approach to $\bQ$-Fano foliations.

\begin{defn}[{See \cite[Definition 3.10]{codim_1_del_pezzo_fols} for details}]\label{defn:log_leaf}
Let $X$ be a normal projective variety, $\sF$ a $\bQ$-Gorenstein algebraically integrable foliation of rank $r$ on $X$, and 
$\eta:\Omega_X^{r} \to \sO_X(K_\sF)$ its corresponding Pfaff field.
Let $F\subset X$ be the closure of a general leaf of $\sF$, and $\tilde e:\tilde F\to X$
the normalization of $F$. Let $m \ge 1$ be the Cartier index of $K_\sF$, \textit{i.e.}, the smallest positive integer $m$ such that $mK_\sF$ is Cartier. 
Then $\eta$ induces a generically surjective map
$\otimes^m\Omega^r_{\tilde F}\to \tilde e^*\sO_X(mK_\sF)$.
Hence there is a canonically defined effective Weil $\bQ$-divisor $\tilde \Delta$ on $\tilde F$ such that 
$mK_{\tilde F}+m\tilde \Delta\sim_\bZ \tilde e^*mK_\sF$. 

We call the pair $(\tilde F,\tilde \Delta)$ a \emph{general log leaf} of $\sF$.
\end{defn}

The next lemma gives sufficient conditions under which the support of $\tilde \Delta$ is precisely the 
image in $\tilde F$ of the singular locus of $\sF$. It is an immediate consequence of \cite[Lemma 5.6]{fano_fols}.

\begin{lemma}\label{lemma:singular_locus}
Let $\sF$ be  an algebraically integrable foliation on a complex projective manifold $X$. 
Suppose that $\sF$ is locally free along the closure of a general leaf $F$.
Let $\tilde{e} : \tilde{F} \to X$ be its normalization, and $(\tilde{F},\tilde{\Delta})$ the corresponding log leaf. 
Then $\Supp(\tilde{\Delta})=\tilde{e}^{-1}(F\cap \textup{Sing}(\sF))$.
\end{lemma}

\begin{defn}\label{def:sing_fol}
Let $X$ be normal projective variety, $\sF$ a $\bQ$-Gorenstein algebraically integrable foliation  on $X$,
and $(\tilde F,\tilde \Delta)$ its general log leaf. 
We say that $\sF$ has \emph{log canonical singularities along a general leaf} if
$(\tilde F,\tilde \Delta)$ is log canonical.
\end{defn}

\begin{rem}
In \cite[Definition I.1.2]{mcquillan08}, McQuillan introduced a notion of log canonicity 
for foliations, without requiring algebraic integrability. 
If a $\bQ$-Gorenstein algebraically integrable foliation $\sF$ 
is log canonical in the sense of McQuillan, then $\sF$
has log canonical singularities along a general leaf
(see \cite[Proposition 3.11]{fano_fols} and its proof). 
\end{rem}

\begin{say}[The family of log leaves] \label{family_leaves} 
Let $X$ be normal projective variety, and $\sF$ an algebraically integrable foliation on $X$.
We describe the \emph{family of leaves} of $\sF$
(see \cite[Lemma 3.2 and Remark 3.8]{fano_fols} for details).
There is a unique irreducible closed subvariety $W$ of $\Chow(X)$ 
whose general point parametrizes the closure of a general leaf of $\sF$
(viewed as a reduced and irreducible cycle in $X$).
It comes with a universal cycle $U \subset W\times X$ and morphisms:

\centerline{
\xymatrix{
U \ar[r]^{e}\ar[d]_{\pi} & X \ , \\
 W &
}
}
\noindent where $e:U\to X$ is birational and, for a general point $w\in W$, 
$e\big(\pi^{-1}(w)\big) \subset X$ is the closure of a leaf of $\sF$.

The variety $W$ is called the \emph{family of leaves} of $\sF$.

Suppose moreover that $\sF$ is $\bQ$-Gorenstein, denote by $m \ge 1$ the Cartier index of $K_\sF$, 
by $r$ the rank of $\sF$, and 
by $\eta:\Omega_X^{r} \to \sO_X(K_\sF)$ the corresponding Pfaff field.
 Given a morphism $V\to W$ from a normal variety,
let $U_V$ be the normalization of $U\times_VW$, with induced morphisms: 

\centerline{
\xymatrix{
U_V \ar[r]^{e_V}\ar[d]_{\pi_V} & X \ . \\
 V &
}
}
\noindent Then $\eta$ induces a generically surjective map
$
\otimes^m\Omega_{ U_V/ V}^{r}\to {e_V}^*\sO_X(mK_\sF).
$
Thus there is a canonically defined effective Weil $\bQ$-divisor $\Delta_V$ on $ U_V$ such that 
$\det(\Omega_{{U_V}/ {V}}^1)^{\otimes m}[\otimes] \sO_{{U_V}}(m\Delta_V) \cong  
{e_V}^*\sO_X(mK_\sF)$. 
Suppose that $v\in V$ is mapped to a general point of $W$, 
set $U_v := (\pi_V)^{-1}(v)$, and $\Delta_v:=(\Delta_V)_{|U_v}$.
Then $(U_v, \Delta_v)$ coincides with the general log leaf $(\tilde F,\tilde \Delta)$ defined above.
\end{say}

\subsection{$\bQ$-Fano foliations}

\begin{defn}\label{def:fano_foliation}\label{def:index_fano_foliation}\label{def:fano_pfaff}\label{def:index_fano_pfaff}
Let $X$ be a normal projective variety, and 
$\sF$ a $\bQ$-Gorenstein foliation on $X$.
We say that $\sF$ is a \emph{$\bQ$-Fano foliation} if $-K_{\sF}$ is ample.
In this case, the \emph{index} of $\sF$ is the largest positive rational number $\iota_{\sF}$ such that
$-K_{\sF} \sim_\bQ \iota_{\sF} H$ for a Cartier divisor $H$ on $X$.
\end{defn}

If $\sF$ is a $\bQ$-Fano  foliation  of rank $r$ on a normal projective variety $X$, then, 
by \cite[Corollary 1.2]{hoeringKOfoliations}, $\iota_{\sF} \le r$.
Moreover, equality holds  if and only if $X$ is a generalized normal cone over a normal projective variety $Z$, 
and $\sF$ is induced by the natural rational map $X \dashrightarrow Z$ 
(see also \cite[Theorem 1.1]{adk08}, and \cite[Theorem 4.11]{codim_1_del_pezzo_fols}).

\begin{defn}
A $\bQ$-Fano foliation $\sF$  of rank $r\ge 2$
is called a \textit{del Pezzo foliation} if  $\iota_{\sF} = r-1$.
\end{defn}

In \cite[Proposition 5.3]{fano_fols}, we proved that
algebraically integrable Fano foliations having log canonical singularities along a general leaf
have a very special property: there is a common point contained in the closure of a general leaf. We
strengthen this result in Proposition~\ref{lemma:common_lc_center} below. 
It will be a consequence of the following theorem.

\begin{thm}[{\cite[Theorem 3.1]{adk08}}]\label{thm:-KX/Y_not_ample}
  Let $X$ be a normal projective variety, $f:X\to C$ a surjective morphism onto a smooth
  curve, and $\Delta$ an effective Weil
  $\bQ$-divisor on $X$ such that $(X,\Delta)$ is log canonical over the generic
  point of $C$. Then $-(K_{X/C}+\Delta)$ is not ample.
\end{thm}

\begin{prop}\label{lemma:common_lc_center}
Let $\sF$ be an algebraically integrable $\bQ$-Fano foliation on a normal projective variety $X$,
having log canonical singularities along a general leaf.
Then there is a closed  irreducible subset $T \subset X$ satisfying the following property.
For a general log leaf $(\tilde{F},\tilde{\Delta})$ of $\sF$, there exists a log canonical center 
$S$ of $(\tilde{F},\tilde{\Delta})$ whose image in $X$ is $T$.
%Let $(\tilde{F},\tilde{\Delta})$ be a general log leaf, and denote by 
%$\tilde{e} : \tilde{F} \to X$ the normalization morphism. There exists a log canonical center 
%$S$ of $(\tilde{F},\tilde{\Delta})$ such that $\tilde{e}(S)=T$.
\end{prop}

\begin{proof}
Let $W$ be the normalization of the family of leaves of $\sF$, 
$U$ the normalization of the universal cycle over $W$,
with universal family morphisms $\pi : U \to W$ and $e : U \to X$.
As explained in \ref{family_leaves}, 
there is a canonically defined effective $\bQ$-Weil divisor $\Delta$ on $U$ such that 
$\det(\Omega_{U/ W}^1)^{\otimes m}[\otimes] \sO_{U}(m\Delta) \cong  
{e}^*\sO_X(mK_\sF)$, where $m \ge 1$ denotes the Cartier index of $K_\sF$. 
Moreover, there is a smooth dense open subset $W_0\subset W$ with the following properties. 
For any $w\in W_0$, denote by $U_w$ the fiber of $\pi$ over $w$, and set $\Delta_w:={\Delta}_{|U_w}$.
Then 
\begin{itemize}
\item $U_w$ is integral and normal, and
\item $(U_w, \Delta_w)$ has log canonical singularities.
\end{itemize}

\smallskip

To prove the proposition, suppose to the contrary that, for any two general points $w,w'\in W_0$, and any
log canonical centers $S_w$ and $S_{w'}$ of $(U_w,\Delta_w)$ and $(U_{w'},\Delta_{w'})$ respectively, we
have $e(S_w) \neq e(S_{w'})$.

\smallskip

Let $C \subset W$ be a (smooth) general complete intersection curve, and $U_C$
the normalization of $\pi^{-1}(C)$,
with induced morphisms $\pi_C: U_C\to C$ and $ e_C: U_C\to X$.
By \cite[Theorem 2.1']{bosch95},
after replacing $C$ with a finite cover if necessary, we 
may assume that $\pi_C$ has reduced fibers. 
As before, there is a 
canonically defined $\bQ$-Weil divisor
$\Delta_C$  
on $U_C$ such that
$K_{U_C/C}+\Delta_C \sim_\bQ  
{e_C}^*K_\sF$. Therefore  
$K_{U_C}+\Delta_C\sim_\bQ\pi_C^*K_{C}+{e}_C^*K_\sF$ is a $\bQ$-Cartier divisor.
For a general point $w \in C$, we identify $\Big(\pi_C^{-1}(w),{\Delta_C}_{|\pi_C^{-1}(w)}\Big)$
with $(U_w,\Delta_w)$, which is log canonical by assumption.
Thus, by inversion of adjunction (see \cite[Theorem]{kawakita07}), 
the pair 
$(U_C,\Delta_C)$ has log canonical singularities over the generic point of $C$. 
Let $w \in C$ be a general point, and 
$S_w$ any log canonical center of
$(U_w,\Delta_w)$. Then there exists a reduced and irreducible
closed subset $S_C \subset U_C$ such that:
\begin{itemize}
\item $S_w={S_C}\cap U_w$, and
\item $S_C$ is a log canonical center of 
$(U_C,\Delta_C)$ over the generic point of $C$.
\end{itemize} 
Moreover, our current assumption implies that 
\begin{itemize}
\item $\dim(e_C(S_C)) = \dim(S_C)$.
\end{itemize} 
Thus, by \cite[Proposition 7.2(ii)]{demailly97}, there exist 
an ample $\bQ$-divisor $A$ and 
an effective $\bQ$-Cartier $\bQ$-divisor $E$ on $U_C$
such that:
\begin{itemize}
\item $e_C^*(- K_{\sF})\sim_{\bQ}A+E$, and
\item for a general point $w \in C$, $\Supp(E)$ does not contain any 
log canonical center of $(U_w,\Delta_w)$.
\end{itemize}
Therefore $(U_C,\Delta_C+\epsilon E)$ is log canonical over the generic point of $C$
for $0<\epsilon \ll 1$.
Notice that $e_C^*(- K_{\sF})-\epsilon E$ is ample
since $e_C^*(- K_{\sF})$ is nef and big, and hence 
$$
 -(K_{U_C/C}+\Delta_C+\epsilon E) \sim_{\bQ} e_C^*(- K_{\sF})-\epsilon E
$$
is ample as well.
But this contradicts Theorem~\ref{thm:-KX/Y_not_ample}, completing the proof of the proposition.
\end{proof}

\begin{cor} \label{Delta_not_0}
Let $\sF$ be an algebraically integrable Fano foliation on a complex projective manifold,
and  $(\tilde{F},\tilde{\Delta})$ its general log leaf.
Suppose that $\sF$ is locally free along the closure of a general leaf.
Then $\tilde{\Delta}\neq 0$.
\end{cor}

\begin{proof}
Denote by $F$ the closure of a general leaf of $\sF$.
If $\tilde{\Delta}=0$, then
$\sF$ is regular along $F$ by Lemma \ref{lemma:singular_locus}. 
Hence $\sF$ is induced by an almost proper map $X\map W$, and $F$ is smooth. 
In particular $(\tilde{F},\tilde{\Delta})$ is log canonical. 
But this contradicts  Proposition~\ref{lemma:common_lc_center}. 
This proves that $\tilde{\Delta}\neq 0$.
\end{proof}

%%%%%%%%%%%%%%%%%%%%%%%%%%%%%%%%%%%%%%%%%%%%%%%%%%%%%%%%%%%%%%%%%

\subsection{Foliations on $\p^n$}\label{fols_in_p^n} \

\smallskip

The \emph{degree} $\deg(\sF)$ of a foliation $\sF$ of rank $r$ on $\p^n$ is defined as the degree 
of the locus of tangency of $\sF$ with a general linear subspace $\p^{n-r}\subset \p^n$. 
By \ref{q-forms}, a foliation on $\p^n$  of  rank $r$ and degree $d$ is given by a twisted $q$-form
$\omega\in H^0\big(\p^n,\Omega^{q}_{\p^n}(q+d+1)\big)$, where $q=n-r$. 
Thus
$$
d \ =\ \deg(K_{\sF})\ +\ r.
$$

Jouanolou has classified codimension $1$ foliations on $\p^n$ of degree $0$ and $1$.
This has been generalized to arbitrary rank as follows.

\begin{thm} \label{cerveau_deserti} \label{lpt3fold} \
\begin{enumerate}
	\item (\cite[Th\'eor\`eme 3.8]{cerveau_deserti}.) 
		A codimension $q$ foliation  of degree $0$ on $\p^n$
		is induced by a  linear projection $\p^n \dashrightarrow \p^q$.
	\item (\cite[Theorem 6.2]{loray_pereira_touzet_2}.)
		A codimension $q$ foliation $\sF$ of degree $1$ on $\p^n$
		satisfies one of the following conditions. 
		\begin{itemize}
			\item $\sF$ is induced by a dominant rational  map $\p^n\dashrightarrow \p(1^{q},2)$,
				defined by $q$ linear forms $L_1,\ldots,L_q$
				and one quadratic form $Q$; or
			\item $\sF$ is the linear pullback of a foliation on $\p^{q+1}$ induced by a global 
				holomorphic vector field.
		\end{itemize}
\end{enumerate}
\end{thm}

\begin{say}  \label{deg1_fols_in_pn}
Let $\sF$ be a codimension $q$ foliation of degree $1$ on $\p^n$. 

In the first case described in Theorem~\ref{lpt3fold}(2), $\sF$ is induced by the $q$-form on $\mathbb{C}^{n+1}$
\begin{multline*}
\Omega  =  \sum_{i=1}^q(-1)^{i+1}L_idL_1\wedge\cdots\wedge \widehat{dL_i}\wedge\cdots \wedge dL_q\wedge dQ +(-1)^{q}2QdL_1\wedge\cdots\wedge dL_q\\
 =  (-1)^q \Big(\sum_{i=q+1}^{n+1}L_j\frac{\partial Q}{\partial L_i}\Big)dL_1\wedge\cdots\wedge dL_q\\
+\sum_{i=1}^{q}\sum_{j=q+1}^{n+1}(-1)^{i+1}L_i\frac{\partial Q}{\partial L_j}
dL_1\wedge\cdots\wedge \widehat{dL_i}\wedge\cdots \wedge dL_q\wedge dL_j,
\end{multline*}
where 
$L_{q+1},\ldots,L_{n+1}$ are linear forms such that 
$L_{1},\ldots,L_{n+1}$ are linearly independent.
Thus, the singular locus of $\sF$ is the union of the quadric 
$\{L_1=\cdots=L_q=Q=0\}\cong Q^{n-q-1}$ and the linear subspace 
$\{\frac{\partial Q}{\partial L_{q+1}}=\cdots =\frac{\partial Q}{\partial L_{n+1}}=0\}$. 

In the second case described in Theorem~\ref{lpt3fold}(2), the singular locus of $\sF$ is the union of linear
subspaces of codimension at least $2$ containing the center $\p^{n-q-2}$ of the projection.

\end{say}

\subsection{Foliations on $Q^n$}\label{fols_in_Q^n} \

In this subsection we classify del Pezzo foliations on smooth quadric hypersurfaces. 

\begin{prop}\label{lemma:fols_in_Q^n} 
Let $\sF$ be a codimension $q$ del Pezzo foliation on a smooth quadric hypersurface $Q^n \subset \p^{n+1}$. 
Then $\sF$ is induced by the restriction of a linear projection $\p^{n+1} \dashrightarrow \p^{q}$.
\end{prop}

\begin{proof}
If $q=1$, then the result follows from \cite[Theorem 1.3]{codim_1_del_pezzo_fols}.
So we assume from now on that $q \ge 2$.

By \cite[Proposition 7.7]{fano_fols}, $\sF$ is algebraically integrable, and
its singular locus is nonempty by 
\cite[Theorem 6.1]{codim_1_del_pezzo_fols}.

Let $x \in Q^n$ be a point in the  singular locus of $\sF$, and
consider the restriction 
$\varphi : Q^n \dashrightarrow \p^n$
to $Q^n$ of the linear projection 
$\psi: \p^{n+1}\dashrightarrow \p^n$
from $x$.
Let $f: Y \to Q^n$ be the blow-up of $Q^n$ at $x$ 
with exceptional divisor $E\cong \p^{n-1}$, and $g : Y \to \p^n$ the induced morphism. 
Notice that $g$ is the blow-up of $\p^n$ along 
the smooth codimension $2$ quadric $Z=\phi(\textup{Exc}(\varphi))\cong Q^{n-2}$. 
Denote by $H$ the hyperplane of $\p^n$ containing $Z$, and by
$F$ the exceptional divisor of $g$.
Note that $g(E)=H$, and $f(F)$ is the hyperplane section of $Q^n$ cut out by $T_xQ^n$. 
The codimension $q$ del Pezzo foliation $\sF$ 
is defined by a nonzero section $\omega\in H^0\big(Q^n, \Omega_{Q^n}^q(q+1)\big)$
vanishing at $x$.
So it induces a twisted $q$-form
$\alpha \in H^0\big(Y, \Omega_Y^q\otimes f^*\sO_Q(q+1)\otimes \sO_Y(-q E)\big)
\cong H^0\big(Y, \Omega_Y^q\otimes g^*\sO_{\p^n}(q+2)\otimes \sO_Y(- F)\big)$. The restriction
of $\alpha$ to $Y \setminus F$ induces a twisted $q$-form
$\tilde \alpha \in H^0\big(\p^n, \Omega_{\p^n}^q(q+2)\big)$ such that 
$\tilde{\alpha}_z(\vec{v}_1,\vec{v}_2,\ldots,\vec{v}_q)=0$ for any $z \in Z$, 
$\vec{v}_1 \in T_z \p^n$, and $\vec{v}_i \in T_z Z$, $2 \le i\le q$. 
Denote by $\tilde \sF$ the foliation on $\p^n$ induced by $\tilde \alpha$.
There are two possibilities:
\begin{itemize}
	\item Either $\tilde \alpha$ vanishes along the hyperplane $H$ of $\p^n$ containing  $Z\cong Q^{n-2}$,
		and hence $\tilde \sF$ is a degree $0$ foliation on $\p^n$; or  
	\item $\tilde \sF$ is a degree $1$ foliation on $\p^n$, and either 
	$Z$ is contained in the singular locus of $\tilde\sF$, or 
	$Z$ is invariant under $\tilde\sF$.
\end{itemize}

We will show that only the first possibility occurs.
In this case, it follows from Theorem~\ref{cerveau_deserti}(1)
that $\sF$ is induced by the restriction of a linear projection $\p^{n+1} \dashrightarrow \p^{q}$.

Suppose to the contrary that $\tilde \sF$ is a degree $1$ foliation on $\p^n$, 
and either $Z$ is contained in the singular locus of $\tilde\sF$, or $Z$ is invariant under $\tilde\sF$.
Recall the  description of the two types of codimension $q$ degree $1$ on $\p^n$
from Theorem~\ref{lpt3fold}(2):
\begin{enumerate}
\item Either the foliation is induced by a dominant rational  map 
$\p^n\dashrightarrow \p(1^{q},2)$,
defined by $q$ linear forms $L_1,\ldots,L_q$
and one quadratic form $Q$; or
\item it is the linear pullback of a foliation on $\p^{q+1}$ induced by a global 
holomorphic vector field.
\end{enumerate} 
In case (2),  the closure of the leaves and the singular locus are all cones with vertex $\p^{n-q-2}$.
Since $Z\cong Q^{n-2}$ is a smooth quadric, we conclude that $\tilde\sF$ must be of type (1),
$Z$ is invariant under $\tilde\sF$, and $\tilde \alpha$ is as in the description of $\Omega$ in \ref{deg1_fols_in_pn}.

Since $Z$ is invariant under $\tilde{\sF}$, we must have 
$\{L_1=\cdots=L_q=Q=0\}\cong Q^{n-q-1}\subset Z$. 
We assume without loss of generality that $H=\{L_1=0\}$. 
Notice that $\{L_1=\cdots=L_q=Q=0\}\subsetneq Z$ since $q\ge 2$.
Let $L_{q+1},\ldots,L_{n+1} \in \mathbb{C}[t_1,\ldots,t_{n+1}]$ be linear forms 
such that $L_{1},\ldots,L_{n+1}$ are linearly independent.
Since $Z$ is invariant under $\tilde{\sF}$, $\varphi^*\tilde{\alpha}$ vanishes identically along 
$f(F)=\{\varphi^*L_1=0\}$. 
It follows from the description of the singular locus of $\tilde \sF$ in \ref{deg1_fols_in_pn}
that we must have 
$\{\varphi^*\frac{\partial Q}{\partial L_{q+1}}=\cdots =
\varphi^*\frac{\partial Q}{\partial L_{n+1}}=0\}=\{\varphi^*L_1=0\}$.
Hence, for $i\in\{q+1, \ldots , n+1\}$,
$\varphi^*\frac{\partial Q}{\partial L_{i}}=a_i\varphi^*L_1$
for some complex number 
$a_i\in\mathbb{C}$.
Then $\psi^*\tilde{\alpha}\in (\psi^*L_1) \cdot H^0\big(\p^{n+1}, \Omega_{\p^{n+1}}^q(q+1)\big)
\subset H^0\big(\p^{n+1}, \Omega_{\p^{n+1}}^q(q+2)\big)$. Therefore, $\tilde{\sF}$
is induced by a degree $0$ foliation on $\p^{n+1}$. So $\tilde{\sF}$ itself is a degree $0$ foliation on $\p^{n}$,
contrary to our assumption.
This completes the proof of the proposition.
\end{proof}

%
%      SECTION 4 
%

\section{{Proof of Theorem \ref{thm:main}}}

Let $X\not\cong \p^n$  be an $n$-dimensional Fano manifold with $\rho(X)=1$, 
and $\sF$ a del Pezzo foliation of rank $r\ge 3$ on $X$.
By Theorem~\ref{thma}, $\sF$ is algebraically integrable. 
Let $F$ be the closure of a general leaf of $\sF$, $\tilde{e} : \tilde{F} \to X$ its normalization, and $(\tilde{F},\tilde{\Delta})$ the corresponding log leaf. 
By assumption, $(\tilde{F},\tilde{\Delta})$ is log canonical, and $\sF$ is  locally free along $F$.

Let $\sA$ be  an ample line bundle on $X$ such that $\textup{Pic}(X)=\bZ[\sA]$, and set $\sL:=\tilde{e}^*\sA$. 
Then $\det(\sF)\cong\sA^{r-1}$, and 
$$
-(K_{\tilde{F}}+\tilde{\Delta}) \ \sim_\mathbb{Z} \ -\tilde{e}^*K_\sF \ \sim_\mathbb{Z} \ (r-1)c_1(\sL).
$$
By Corollary~\ref{Delta_not_0}, $\tilde{\Delta}\neq 0$.
So we can apply Theorem \ref{thm:classification_del_pezzo}.
Taking into account that 
if $\tilde F$ is singular, then its singular locus is contained in the support 
of $\tilde \Delta$ by Lemma~\ref{lemma:singular_locus}, 
we get the following possibilities 
for the triple $(\tilde{F},\sL,\tilde{\Delta})$:
\begin{enumerate}
\item $\big(\tilde{F},\sL,\sO_{\tilde{F}}(\tilde\Delta)\big)\cong (\p^r,\sO_{\p^r}(1),\sO_{\p^r}(2))$. 
\item $\big(\tilde{F},\sL,\sO_{\tilde{F}}(\tilde\Delta)\big) \cong (Q^r,\sO_{Q^r}(1),\sO_{Q^r}(1))$, 
where $Q^r$ is a smooth quadric hypersurface in $\p^{r+1}$.
\item $r=1$ and $\big(\tilde{F},\sL,\sO_{\tilde{F}}(\tilde\Delta)\big)\cong (\p^1,\sO_{\p^1}(d),\sO_{\p^1}(2))$ for some integer $d\ge 3$.
\item $r=2$ and $\big(\tilde{F},\sL,\sO_{\tilde{F}}(\tilde\Delta)\big)\cong(\p^2,\sO_{\p^2}(2),\sO_{\p^2}(1))$.
\item $(\tilde{F},\sL)\cong (\p_{\p^1}(\sE),\sO_{\p_{\p^1}(\sE)}(1))$  for an ample vector bundle $\sE$ 
of rank $r$ on $\p^1$. 
Moreover, one of the following holds.
\begin{enumerate}
\item $r=2$ and $\sE=\sO_{\p^1}(1)\oplus\sO_{\p^1}(a)$ for some $a\ge 2$, and $\tilde\Delta\sim_\mathbb{Z} \sigma+ f$
where $\sigma$ is the minimal section and $f$ a fiber of $\p_{\p^1}(\sE) \to \p^1$.
\item $r=2$ and $\sE=\sO_{\p^1}(2)
\oplus\sO_{\p^1}(a)$ for some $a\ge 2$, and $\tilde\Delta$ is a minimal section.
\item $r=3$ and $\sE=\sO_{\p^1}(1)\oplus \sO_{\p^1}(1)
\oplus\sO_{\p^1}(a)$ for some $a \ge 1$, and $\tilde{\Delta}=\p_{\p^1}(\sO_{\p^1}(1)\oplus \sO_{\p^1}(1))$.
\end{enumerate}
\item $\sL$ is very ample, and embeds $(\tilde{F},\tilde{\Delta})$ as a cone over $\big((Z,\sM),(\Delta_Z, \sM_{|\Delta_Z})\big)$,
where $(Z,\sM,\Delta_Z)$ satisfies one of the conditions (2-5) above. 
\end{enumerate}

\medskip

First we show that case (1) cannot occur. 
Suppose otherwise that $(\tilde{F},\sL,\sO_{\tilde{F}}(\tilde{\Delta}))\cong (\p^r,\sO_{\p^r}(1),\sO_{\p^r}(2))$.
By Proposition~\ref{lemma:common_lc_center},
there is a common point $x$ contained in the closure of a general leaf.
Since $(\tilde{F},\sL)\cong (\p^r,\sO_{\p^r}(1))$, 
there is an irreducible $(n-1)$-dimensional family of lines on $X$ through $x$
sweeping out the whole $X$. 
Lemma~\ref{lemma:bound_index} together with Theorem~\ref{kobayashi_ochiai} imply that 
$X\cong \p^n$, contrary to our assumptions. 

\medskip

Next suppose that we are in case (2) or (5c).
Note that $\tilde\Delta$ is irreducible in either case (in case (2), $\tilde F$ is a smooth quadric of dimension $r\ge 3$ and $\tilde\Delta$ is a hyperplane section).
By  Proposition~\ref{lemma:common_lc_center}, the image $T$ of $\tilde\Delta$ in $X$ is contained in the closure of a general leaf of $\sF$.
There is a family of lines on $X$, all contained in leaves of $\sF$ and meeting $T$,
that  sweep out the whole $X$. 
In case (2), this corresponds to the family of lines on $\tilde F\cong Q^r$.
In case (5c), it corresponds to the family of lines on fibers of $\tilde F\to \p^1$.
Let $x\in T$ be a general point. 
Then there is an irreducible $(n-2)$-dimensional family of lines on $X$ through $x$, and the general line in this family is 
free by Remark~\ref{rem:general_line_is_free}.
By Lemma~\ref{lemma:bound_index}, $\iota_X\ge n$.
Theorem~\ref{kobayashi_ochiai} then implies that $X\cong Q^n$.
By Proposition~\ref{lemma:fols_in_Q^n},
 $\sF$ is induced by the restriction to $Q^n$ of a linear projection $\p^{n+1} \dashrightarrow \p^{n-r}$.

\medskip
 
 Cases (3), (4), (5a) and (5b) do not occur since we are assuming $r\ge 3$. 
 
\medskip

Finally suppose that we are in case (6): $\sL$ is very ample, and embeds $(\tilde{F},\tilde{\Delta})$ 
as a cone over the pair $\big((Z,\sM),(\Delta_Z, \sM_{|\Delta_Z})\big)$,
where  $(Z,\sM,\Delta_Z)$ satisfies one of the conditions (2-5) above. 
As in the proof of Theorem~\ref{thm:classification_del_pezzo}, set $m:=\dim(Z)$, $s:=r-m$, 
and let $e : Y \to \tilde F$ be the blow-up of $\tilde F$ along its vertex, with exceptional divisor $E$. 
Then $Y\cong \p_Z(\sM\oplus \sO_Z^{\oplus s})$, with natural projection $\pi : Y\to Z$. 
Moreover the strict transform $\Delta_Y$ of $\tilde \Delta$ in $Y$
satisfies $\Delta_Y=\pi^*\Delta_Z$.
A straightforward computation gives 
$$
K_Y+\pi^*\Delta_Z\sim_\mathbb{Z}e^*(K_{\tilde{F}}+\tilde{\Delta})+(m-2) E.
$$
On the other hand, by \cite[8.3]{fano_fols}, there exists an effective divisor $B$ on $Y$
such that 
$$K_Y+E+B\sim_\mathbb{Z} e^*(K_{\tilde{F}}+\tilde{\Delta}).$$
Therefore
$$\Delta_Y=\pi^*\Delta_Z\sim_\mathbb{Z}(m-1)E+B.$$
We conclude that $m=1$.
Thus $\tilde F$ is isomorphic to a cone with vertex $V\cong \p^{r-2}$
over a rational normal curve of degree $d \ge 2$,  
and $\tilde{\Delta}$ is the union of two rulings $\Delta_1$ and $\Delta_2$, each isomorphic to $\p^{r-1}$.

By Proposition~\ref{lemma:common_lc_center},  
there is a log canonical center $S$ of $(\tilde{F},\tilde{\Delta})$ whose image in $X$
does not depend on the choice of the general log leaf. 
Either $S=V$, or $S=\Delta_i$ for some $i\in\{1,2\}$. 
If $S=V$, then the lines  through a general point of $\tilde e(V)$ sweep out the whole $X$. 
Lemma~\ref{lemma:bound_index} together with Theorem~\ref{kobayashi_ochiai} then imply that 
$X\cong \p^n$, contrary to our assumptions. 
We conclude that the image of $V$ in $X$ varies with $(\tilde{F},\tilde{\Delta})$,
and, for some $i\in\{1,2\}$, $T=\tilde e(\Delta_i)$ is contained in the closure of a general leaf.
There is a family of lines on $X$, all contained in leaves of $\sF$ and meeting $T$,
that  sweep out the whole $X$. 
Let $x\in T$ be a general point. 
Since $V\subset \Delta_i$, and the  image of $V$ in $X$ varies with $(\tilde{F},\tilde{\Delta})$,
there is an irreducible $(n-2)$-dimensional family of lines on $X$ through $x$, and the general line in this family is 
free by Remark~\ref{rem:general_line_is_free}.
By Lemma~\ref{lemma:bound_index}, $\iota_X\ge n$.
Theorem~\ref{kobayashi_ochiai} then implies that $X\cong Q^n$.
By Proposition~\ref{lemma:fols_in_Q^n},
 $\sF$ is induced by the restriction of a linear projection $\p^{n+1} \dashrightarrow \p^{n-r}$.
\qed

\medskip

Using Theorem~\ref{nakayama} and the 
same arguments as in the proof of Theorem \ref{thm:main}, 
one can get the following result for del Pezzo foliations of rank $2$,
without the assumption that $\sF$ is log canonical along  a general leaf.

\begin{prop}\label{classification_r=2}
Let $\sF$ be a del Pezzo foliation of rank $2$ on a complex projective manifold 
$X\not\cong \p^n$ with $\rho(X)=1$, and  suppose that $\sF$ is locally free along  a general leaf. 
Denote by $(\tilde F,\tilde \Delta)$ the general log leaf  of  $\sF$, and by 
$\sL$ the pullback to $\tilde F$ of the ample generator of $\Pic(X)$.
Then the triple $(\tilde F,\sO_F(\tilde\Delta),\sL)$ is isomorphic to one of the following.
\begin{enumerate}
\item $\big(\p^2,\sO_{\p^2}(1),\sO_{\p^2}(2)\big)$;
\item $\big(\p^2,\sO_{\p^2}(2),\sO_{\p^2}(1)\big)$;
\item $\big(\mathbb{F}_d,\sO_{\mathbb{F}_d}(\sigma+f),\sO_{\mathbb{F}_d}(\sigma+(d+1)f)\big)$, 
where $d\ge 0$, $\sigma$ is a minimal section, and $f$ is a fiber of $\mathbb{F}_d \to \p^1$; 
\item $\big(\p(1,1,d),\sO_{\p(1,1,d)}(2\ell),\sO_{\p(1,1,d)}(d\ell)\big)$, 
where $d\ge 2$, and $\ell$ is a ruling of the cone  $\p(1,1,d)$.
\end{enumerate}
\end{prop}

\begin{rem}
There are examples of del Pezzo foliations of rank $2$ on Grassmannians whose general log leaves
are $(\tilde F,\tilde \Delta)= (\p^2,2\ell)$ (see \cite[4.3]{fano_fols}).
\end{rem}

\bibliographystyle{amsalpha}
\bibliography{foliation}

\end{document}